\documentclass[12pt, singlespace]{amsart}
\usepackage{amsmath,amssymb,amsthm,mathrsfs}
\usepackage[all]{xy}

\newcommand\C{\mathbb C}

\newcommand\Q{\mathbb Q}
\newcommand\R{\mathbb R}

\newcommand\A{\mathbb A}
\renewcommand\P{\mathbb P}
\renewcommand\O{\mathcal O}

\newcommand{\Spec}{\operatorname{Spec}}
\newcommand{\DF}{\operatorname{DF}\,}
\newcommand{\CM}{\operatorname{CM}\,}
\newcommand{\Bl}{\operatorname{Bl}}
\newcommand{\init}{\operatorname{in}}
\newcommand\0{\{0\}}

\newcommand\cB{\mathcal B}
\newcommand\cZ{\mathcal Z}
\newcommand\cW{\mathcal W}
\newcommand\cR{\mathcal R}
\newcommand\cV{\mathcal V}
\newcommand\cF{\mathcal F}
\newcommand\cT{\mathcal T}
\newcommand\cN{\mathcal N}
\newcommand\cM{\mathcal M}

\newcommand\X{\mathcal X}

\renewcommand\L{\mathcal L}

\theoremstyle{definition}
\newtheorem{defn}{Definition}[section]
\newtheorem{rem}[defn]{Remark}

\theoremstyle{plain}

\newtheorem{thm}[defn]{Theorem}
\newtheorem{lem}[defn]{Lemma}
\newtheorem{prop}[defn]{Proposition}
\newtheorem{conj}[defn]{Conjecture}

  \makeatletter
    
    \@addtoreset{equation}{section}
  \makeatother

\title[Minimizing CM degree and slope stability]
{\textbf{Minimizing CM degree and slope stability of projective varieties}}
\author{Kentaro Ohno}
\address{Graduate School of Mathematical Sciences, University of Tokyo, 3-8-1, Komaba,
Meguro-ku, Tokyo, 153-8914, Japan}
\email{ohken322@gmail.com}

\date{\today}

\begin{document}


\begin{abstract}
We discuss a minimization problem of the degree of the CM line bundle among all possible fillings of a polarized family with fixed general fibers. We show that such minimization implies the slope semistability of the fiber if the central fiber is smooth.
\end{abstract}

\maketitle


\section{introduction}

We work over the complex number field $\C$ throughout this paper.
By a \emph{polarized variety} $(V,L)$, we mean a pair of a projective variety $V$ and an ample $\Q$-line bundle $L$ on $V$.
A \emph{polarized family} $ (\X,\L)\to C$ consists of a smooth projective curve $C$, a variety $\X$ with a projective flat morphism $\X \to C$, and a relatively ample $\Q$-line bundle $\L$ on $\X$.

The minimization of the degree of the CM line bundle (which we call the \emph{CM degree} or the \emph{DF invariant}) in a certain class of polarized families was considered in \cite{WX14} in the context of compactification problem of moduli space.
They observed that for a family of canonically polarized varieties with semi-log canonical singularities (named \emph{KSBA-stable} family after Koll\'{a}r-Shepherd-Barron \cite{KSB88}, and Alexeev \cite{Ale94}) over a punctured curve, the KSBA-stable compactification indeed minimizes the CM degree.
Moreover, \cite{Od13b} proved similar statements for families of other classes of polarized varieties such as Calabi-Yau varieties and Fano varieties with large alpha-invariant which are known to be K-stable.
By this observation, we expect that the K-stable compactification of Fano families should minimize the CM degree, which leads to the separatedness of \emph{K-moduli} \cite{Od13b}. 
Furthermore, in a private communication, Odaka told the author about the following conjecture which seems not yet to appear in the literature.

\begin{conj}[Odaka]\label{conj}
Let $(\X,\L = -K_\X) \to C$ be a family
polarized by the anti-canonical class over a smooth curve $C$ with a fixed closed point $0 \in C$.
Then, the fiber $\X_0$ over $0 \in C$ is K-semistable if and only if $ \CM (\X_{C'},\L_{C'}) \leq \CM (\X',\L') $ holds for any pointed curve $(C',0') \to (C,0)$ and any polarized family $(\X',\L') \to C'$ which is isomorphic to $(\X_{C'},\L_{C'})$ over $C'\setminus \{0'\}$, where $(\X_{C'},\L_{C'}) = (\X,\L) \times_C C'$.
Moreover, the strict inequality holds for any normal $(\X',\L')$ which is not isomorphic to the normalization of $(\X_{C'},\L_{C'})$ if and only if $\X_0$ is K-stable.
\end{conj}

We give a remark on the existence of the filling which minimizes the CM degree.
It has been expected (and recently proved in \cite{BLX}) that the K-semistability is the open condition.
So Conjecture \ref{conj}, if true, implies that the CM-minimizer does not exist if a general fiber is not K-semistable.

The aim of this paper is to investigate the relation between the minimization problem of the CM degree and the K-stability to approach the above conjecture. 
In particular, our main theorem is the following.

\begin{thm}\label{thm1}
Let $ (\X,\L) \to C $ be a polarized family and $ (\X_0,\L_0)$ be the fiber over a closed point $ 0 \in C$.
Assume that $\X_0$ is a smooth variety and that the inequality $ \CM (\X,\L) \leq \CM (\X',\L') $ holds for any polarized family $(\X',\L') \to C$ isomorphic to $ (\X,\L) $ over $C\setminus \0$. 
Then, the fiber $(\X_0,\L_0)$ is slope semistable.
\end{thm}

The notion of the slope stability of polarized varieties is introduced in \cite{RT07} as a weak version of the K-stability, that is, the K-stability for a special class of test configurations obtained by a deformation to the normal cone.
In comparison with Conjecture \ref{conj}, we note that Theorem \ref{thm1} holds for not only Fano families, but also \emph{any} polarized families, although we assume the smoothness of the central fiber.
Also, note that the minimization assumption in Theorem \ref{thm1} is weaker than that in Conjecture \ref{conj} in the sense that we do not need the minimization over base changes in Theorem \ref{thm1}.

\medskip
\noindent {\bf Sketch of the proof of the main theorem}. 

Let $ Z \subset \X_0 $ be a proper closed subscheme and $ c \in (0, \epsilon(Z, \L_0) ) $ be a rational number, where $\epsilon(Z,\L_0)$ is the Seshadri constant of $Z$ with respect to $\L_0$.
Take the deformation to normal cone over $Z$ 
\[
\pi_Z: \cT_Z = \Bl_{Z \times \0} ( \X_0 \times \A^1) \to \X_0 \times \A^1
\]
polarized by a relatively ample $\Q$-line bundle $ \L_{Z,c} = \pi_Z^* p_1^* \L_0 (- cE_Z) $, where $ p_1 : \X_0 \times \A^1 \to \X_0$ is the first projection and $E_Z$ is the Cartier exceptional divisor. Let $\overline{\cT_Z} \to \P^1$ be the natural compactification of $\cT_Z \ \to \A^1$.
We need to show the inequality $ \DF(\cT_Z, \L_{Z,c}) \geq 0 $ in order to prove slope semistability of the central fiber $ (\X_0, \L_0) $.

To show the inequality, we define another polarized family $(\cB, \cM) \to C$ by
\[
\pi : \cB = \Bl_Z \X \to \X , \ \cM = \pi^* \L ( -c E)
\]
where $E$ is the Cartier exceptional divisor.
We relate the difference of the CM degree $\bigl( \CM(\cB, \cM) - \CM(\X,\L) \bigr)$ to $\DF(\cT_Z, \L_{Z,c})$ by making use of a degeneration technique as follows.
First we take the deformation of $\X$ to the normal cone $\X_0 \times \A^1$ of $\X_0$, that is, blow up $\X \times \A^1$ along $\X_0 \times \0$.
Let $\cZ$ be the strict transform of $ Z \times \A^1 \subset \X \times \A^1$.
Then, the blow-up of the total family along $\cZ$ gives a deformation of $\cB$ to $\overline{\cT_Z} \cup_{\X_0} \X$.
Although there may exist exceptional divisors of the blow-up contained in the central fiber of the deformation in general,
we show that the smoothness of $\X_0$ ensures there are no such exceptional divisors.
Then we have the equality
\[
\CM(\cB,\cM) - \CM(\X,\L) = \DF(\cT_Z, \L_{Z,c})
\]
by the flatness.
Thus, by the using minimizing assumption, we get the inequality
\[
\DF(\cT_Z, \L_{Z,c}) \geq 0
\] 
to reach the conclusion.

	\medskip
	\noindent{\bf A postscript note:}
	Soon after the first version of this article had been posted on the arXiv, the author was informed that Blum and Xu \cite{BX} proved the separatedness of the K-moduli, which had been the original motivation for our study.
	Moreover, they told the author that they proved one direction of Conjecture \ref{conj} and that C.Li and X. Wang independently obtained the same result;  the K-semistable filling implies the CM-minimization.
	
	\medskip
	\noindent{\bf Acknowledgements:}
	The author would like to express great gratitude to his supervisor Yoshinori Gongyo for his continuous encouragement and valuable advice.
	He is also deeply grateful to Yuji Odaka for sharing his problem, for teaching him the backgrounds and the related notions, and for his warm encouragement.
	The author is supported by the FMSP program at the University of Tokyo.

\section{Preliminaries}

The aim of this section is to recall some definitions and related results used in the proof of the main theorem.

\subsection{Test configurations and the DF invariant}
In this subsection, we recall the definition of test configurations and the DF invariant, which appear in the definition of K-stability.
\begin{defn}
A \emph{test configuration} $ (\X, \L) $ for a polarized variety $ (V,L) $ consists of the following data: 
\begin{enumerate}
  \item A variety $\X$ admitting a projective flat morphism $ f :\X \to \A^1 $,
  \item An $f$-ample $\Q$-line bundle $\L$ on $\X$, 
  \item A $ \C^* $-action on $ (\X,\L ) $ compatible with the natural $\C^*$-action on $\A^1$ via $f$,
\end{enumerate}
  such that the restriction $(\X, \L)|_{\C^*}$ over $\C^*$ is $\C^*$-equivariantly isomorphic to $(V,L) \times \C^*$.

If we only assume that $\L$ is $f$-semiample instead of $f$-ample, then $(\X,\L)$ is called a \emph{semi-test configuration}.
A test configuration $ (\X,\L)$ is said to be \emph{trivial} if $\X$ is equivariantly isomorphic to the trivial family $ V \times \A^1 $ with the trivial action on the first factor $V$.
\end{defn}

Given a test configuration $ (\X,\L) $ for an $n$-dimensional polarized variety $ (V,L) $, there is an $ \C^* $-action on $ H^0( \X_0, \L_0^k) $ for a sufficiently divisible positive integer $k$ induced by that on the central fiber $ (\X_0, \L_0 ) $. 
If we decompose the $ \C $-vector space $ H^0( \X_0, \L_0^k) $ into eigenspaces with respect to the action of $\C^*$,
the eigenvalues can be written as some power of $ t \in \C^* $.
We call the exponent as the \emph{weight} of the action on each eigenvector.
The \emph{total weight} $ w(k) $ is the sum of the weight over the eigenbasis.
By the equivariant Riemann-Roch theorem, we have an expansion
\begin{align}
w(k) = w_0 k^{n+1} + w_1 k^n + O(k^{n-1}).	\label{weight}
\end{align}
Also we write an expansion of $ \chi (V,L) $
\[
\chi (V,L^k) = a_0 k^n + a_1 k^{n-1} + O(k^{n-2})
\]
for sufficiently divisible $k$.

\begin{defn} (\cite{Don02})
In the above notation, the \emph{Donaldson-Futaki invariant} for a test configuration $ (\X, \L) $ is defined as
\[
\DF (\X,\L) = a_1 w_0 - a_0 w_1.
\]
\end{defn}

Note that we can naturally extend the definition of the Donaldson-Futaki invariant to arbitrary \emph{semi}-test configurations (see \cite{RT07}).

We do not use the following definition of K-stability in the proof of Theorem \ref{thm1}, but we introduce it to clarify the motivation of our study.

\begin{defn}\label{K-stab}
(\cite{Don02}, see also \cite{Sto11})
A polarized variety $ (V,L) $ is 
\begin{itemize}
  \item \emph{K-semistable} if the Donaldson-Futaki invariant $ \DF (\X,\L) $ is nonnegative for any test configuration $(\X,\L)$ for $(V,L)$.
  \item \emph{K-polystable} if it is K-semistable and $ \DF (\X,\L) =0 $ only if $\X$ is isomorphic to $ V \times \A^1 $ outside some closed subset of codimension at least 2.
  \item \emph{K-stable} if it is K-semistable and $ \DF (\X,\L) =0 $ only if $\X$ is $\C^*$-equivariantly isomorphic to $ V \times \A^1 $ with the natural $\C^*$-action outside some closed subset of codimension at least 2.
\end{itemize}

\end{defn}

Note that we assume non-triviality in codimension 1 of test configurations in the definition of K-(poly)stability  \cite{LX14,Sto11}. If $V$ is normal, we only need to consider non-trivial normal test configurations for K-(semi)stability since the Donaldson-Futaki invariant does not increase by normalization \cite[Remark 5.2]{RT07}.

\subsection{the CM degree}

In this paper, we only need to treat the degree of the CM line bundle over a curve, which we define a priori as follows.
For more details, we refer to \cite{FR06}.

\begin{defn}\label{CM}
For a polarized family $(\X,\L) \to C$ with a fiber $(\X_t,\L_t)$ of dimension $n$,
write 
\begin{align}
\chi (\X_t,\L_t^k) = a_0 k^n + a_1 k^{n-1} + O(k^{n-2})	\label{a}	\\
\chi (\X,\L^k) = b_0 k^{n+1} + b_1 k^n + O(k^{n-1})		\label{b}
\end{align}
for sufficiently divisible positive integer $k$.
The coefficient $ a_i $ is independent of the choice of a fiber since $ \chi $ is constant over a flat family.
Let $g(C)$ denote the genus of $C$.
Then the \emph{CM degree} is defined as
\[
\CM(\X,\L) = a_1 b_0 - a_0 b_1 + (1-g(C))a_0^2.
\]
\end{defn}

This value is nothing but the degree of the CM line bundle $\lambda_{CM}$\cite{PT06,FS90} of $ L $ on $ C $.

\begin{rem}\label{DFisCM}
Given a normal test configuration $(\X,\L) \to \A^1$ for a normal polirized variety $(V,L)$, let 
$
(\overline{\X}, \overline{\L} ) \to \P^1
$
denote the natural $ \C^*$-equivariant compactification, that is, we add the trivial fiber $ (V,L) \times \{ \infty \} $ over $ \infty \in \P^1 $.  
Then it is well known (see for example \cite{Od13, BHJ}) that the total weight $ w(k) $ on $ H^0(\X_0, \L_0 ) $ can be written as
\[
w(k) =\chi (\overline{\X}, \overline{\L}^k) - h^0 (\X_0, \L_0^k).
\]
Using the asymptotic Riemann-Roch formula, we get the equalities
\[
w_0 = b_0, \ w_1 = b_1 - a_0
\]
using the notation in (\ref{weight}).
Thus, the CM degree of $ (\overline{\X}, \overline{\L} ) $ coincides with the Donaldson-Futaki invariant of $ (\X,\L) $.
In this viewpoint, the CM degree is often called the Donaldson-Futaki invariant, too.
\end{rem}

\subsection{Slope stability}
In this subsection, we recall the notion of the slope semistability of polarized varieties introduced in \cite{RT07}. 

Let $ (V, L) $ be an $n$-dimensional polarized variety. Write
\[
\chi ( V, L^k) = a_0 k^n + a_1 k^{n-1} + O(k^{n-2})
\]
for sufficiently divisible positive integer $k$.
Then the \emph{slope of $(V,L)$} is defined as
\[
\mu (V, L) = \frac{a_1}{a_0}.
\]
Let $Z \subset V$ be a proper closed subscheme defined by an ideal $ I_Z $ and take the blow-up along $Z$
\[
\sigma : \hat{V} = \Bl_Z V \to V.
\]
Let $E$ be the Cartier exceptional divisor defined by $I_Z$ on $\hat{V}$ then the \emph{Seshadri constant} $\epsilon ( Z, L) $ of $Z$ with respect to $L$ is defined as
\[
\epsilon = \epsilon (Z,L) := \sup \{ x > 0 \mid \sigma^* L (- xE) : \textrm{ample} \}.
\]
For a  rational number $x \in (0, \epsilon (Z,L)]$, write
\[
\chi ( \hat{V} , (\sigma^* L (-xE))^k) = a_0(x) k^n + a_1(x)k^{n-1} + O(k^{n-2})
\]
for a sufficiently divisible $k$.
Here, $a_0(x)$ and $a_1(x)$ are polynomials of $x$.
Then the \emph{slope along $Z$} with respect to $c \in ( 0, \epsilon ] \cap \Q $ is defined as
\[
\mu_c (I_Z, L) = \frac{ \int^c_0 ( a_1 (x) + \frac{a'_0 (x)}{2} ) dx }{ \int^c_0 a_0 (x) dx }.
\]

\begin{defn} (\cite{RT07}) 
$ (V,L) $ is \emph{slope semistable} if the inequality
\[
\mu (V,L) \geq \mu_c (I_Z, L)
\]
holds for any proper closed subscheme $Z \subset V$ and any rational number $c \in (0,\epsilon] $.
\end{defn}

The slope semistability is a (strictly) weaker notion than the K-semistability as in Theorem \ref{slopeisDF}.
To see this, first take a deformation to the normal cone over $Z$
\[
\pi : \cT_Z = \Bl_{Z \times \0} (V \times \A^1) \to V \times \A^1 
\]
and let $F$ be the Cartier exceptional divisor.
We define a $\Q$-line bundle $\L_{Z,c} := \pi^* p_1^* L (-cF) $ for $c \in (0,\epsilon] \cap \Q$. 

\begin{lem}\label{lem1}
In the above setting,
$\L_{Z,c}$ is ample over $\A^1$ for $c \in (0, \epsilon)$ 
where $p_1 : V \times \A^1 \to V$ is the first projection.
Moreover, $ L_{Z,\epsilon} $ is semiample over $\A^1$ if $ \sigma^* L (-\epsilon E) $ is semiample.
\end{lem}

\begin{proof}
See \cite[Proposition 4.1]{RT07}.
\end{proof}

Thus, we can see $(\cT_Z,  \L_{Z,c} )$ as a (semi-)test configuration of $(V,L)$ for $ c \in (0, \epsilon)$ and for $c = \epsilon$ if $\sigma^* L(-\epsilon E)$ is semiample.

\begin{thm}\label{slopeisDF}
In the above notation, the Donaldson-Futaki invariant $ \DF(\cT_Z, \L_{Z,c} ) $ is a positive multiple of $ \bigl( \mu (V,L) - \mu_c (I_Z,L) \bigr) $ for any rational number $c \in (0, \epsilon )$ and for $c = \epsilon$ if $\sigma^* L(-\epsilon E)$ is semiample.
In particular, $(V,L)$ is slope semistable if it is K-semistable.
\end{thm}

\begin{proof}
See \cite[Section 4]{RT07}.
\end{proof}

\begin{rem}
As in \cite{PR09}, a blow-up of $\P^2$ at two points is slope semistable, although it is not K-semistable. So this example shows that the slope semistability is indeed strictly weaker than the K-semistability.
\end{rem}

\section{Deformation to test configurations}\label{deform}

We fix a polarized family $(\X,\L) \to C$ such that the fiber $(\X_0,\L_0)$ over a fixed closed point $0 \in C$ is a variety.
The aim of this section is to construct a deformation of another polarized family over $\X$ to a test configuration of the central fiber $(\X_0, \L_0)$, and compare their CM degrees.

\subsection{Construction}\label{construct}

We refer to \cite{Ful} for a detailed description of a deformation to the normal cone, which we use for the construction.

First we take a deformation to the normal cone over $ \X_0 $
\[
\sigma : \cV = \Bl_{\X_0 \times \0} ( \X \times \A^1 ) \to  \X \times \A^1 .
\]
Then the central fiber $\cV_0$ of $\cV \to \A^1$ can be written as a union
\[
\cV_0 = \hat{\X} \bigcup_{\X_0} P 
\]
glued along $\X_0$.
Here $\hat{\X} \cong \X$ is the strict transform of $\X \times \0$ and $P$ is the exceptional divisor.
Note that since the normal bundle of $\X_0 \times \0$ is trivial, $P$ is isomorphic to $ \X_0 \times{\P}^1$
and so has a natural $\C^*$-action induced by that on $\P^1$. 
$P$ is glued to $\hat{\X}$ along one of the $\C^*$-invariant fiber $\X_0 \times \{ \infty \} \subset \X_0 \times{\P}^1 \cong P$.
Also we remark that $\cV$ admits a natural flat morphism to a surface $\Bl_{(0,0)}(C\times \A^1)$ by the universal property of blow-ups.

Consider a closed subscheme $Z \subset \X$ set-theoretically supported in $\X_0$.
Let $ \cZ $ be the strict transform of $ Z \times \A^1 \subset \X \times \A^1$ on $\cV$.
Then $ \cZ $ gives a flat degeneration of $Z \subset \X$ to a $\C^*$-invariant closed subscheme $Z_0 \subset P $ by Lemma \ref{degen-ideal} below.
We take the blow-up along $\cZ$
\begin{align}
\Pi : \cW = \Bl_{\cZ} \cV \to \cV	\label{totalfamily}
\end{align}
and let $G$ be the Cartier exceptional divisor. 
Identify the general fiber of $\cV \to \A^1$ with $\X$ and let
\begin{align*}
\pi_0 	&: \cT = \Bl_{Z_0} P \to P,	\\
\pi 	&: \cB = \Bl_{Z} \X \to \X
\end{align*}
denote the strict transform of $P$ and $\cV_t \cong \X$ on $\cW$ respectively, 
and 
\begin{align*}
E_0 	&= G |_\cT,	\\
E 	&= G |_\cB
\end{align*}
be each Cartier exceptional divisor.
We have the following diagram:

\[
\xymatrix@M=8pt{
E \subset	\cB	\ar[r]^-\pi \ar@{^{(}-_>}@<2.5ex>[d]	& \X  \ar@{=}[r]	\ar@{^{(}-_>}[d]	& \X \times \{t\}
\ar@{^{(}-_>}[d]	\\
G \subset	\cW  	\ar[r]^-\Pi				& \cV \ar[r]	^-\sigma				& \X \times \A^1
\ar[r]^-{q_1}	& \X		\\
E_0 \subset	\cT	\ar[r]^-{\pi_0} \ar@{_{(}-_>}@<-2.5ex>[u]	& P    \ar[r]	^-{p_1}\ar@{_{(}-_>}[u]	& \X_0 \times \0  \ar@{_{(}-_>}[u]
}
\]

We fix a positive rational number $c$ and define a $\Q$-line bundle $ \cF := ( \Pi^* \sigma^* q_1^* \L ) (-cG) $ on $\cW$
where $ q_1 : \X \times \A^1 \to \X $ is the first projection.
Then, by taking restriction to each component of fibers, we have
\begin{align*}
\cF |_\cT 	&= ( \pi_0^* p_1^*\L_0 ) (-cE_0) =: \cN,	\\
\cF |_{\hat{\X}} &= \L , 	\\
\cF |_\cB 	&= \pi^* \L (-cE) =: \cM ,	
\end{align*}
where $ p_1 : \X_0 \times \P^1 \to \X_0 $ is the first projection.
Thus, when $c$ is sufficiently small, 
$(\cT, \cN)$ can be seen as a test configuration for $(\X_0, \L_0)$ and the general fiber $(\cB, \cM)$ of $(\cW , \cF) \to \A^1$ is a polarized family.

Next we show how we can treat the above deformation algebraically (see also \cite{LX16,LZ18}).
Let
\begin{align*}
\cR 	&= \O_\X[t,I_{\X_0}t^{-1}] 		\\
	&= \O_\X[t] + I_{\X_0} t^{-1} + I_{\X_0}^2 t^{-2} + \cdots \subset \O_\X [t, t^{-1}]
\end{align*}
be the extended Rees algebra (see \cite[6.5]{Eis94}) of the ideal $I_{\X_0} \subset \O_\X$ defining $\X_0$.
Then, as in \cite[Lemma 4.1]{LZ18} we have isomorphisms of $ \O_\X $-algebras
\begin{align}
\cR \otimes_{\C[t]} \C[t, t^{-1}]	&\cong  \O_\X [t, t^{-1}]	, 	\notag	\\
\cR \otimes_{\C[t]} \C[t]/(t)		&\cong  \bigoplus_{k \geq 0} ( I_{\X_0}^k / I_{\X_0}^{k+1} )	\cong	 \O_{\X_0} [s], \label{isom1}
\end{align}
so that we can discribe the above deformation algebraically as
\[
\cV^{\circ}:= \cV \setminus \hat{\X}  = \Spec_\X \cR \to \mathbb{A}^1_t
\]
with the central fiber
\[
\cV^{\circ}_0 = \cV_0 \setminus \hat{\X} = \X_0 \times \A^1_s.
\]
Let $I \subset \O_\X$ be a sheaf of ideals which defines a subscheme supported in (the thickening of) $\X_0$.
For a non-zero local section $f$ of $I$ defined around the generic point of $\X_0$, let $k = \textrm{ord}_{\X_0}(f)$ be the minimum integer such that $ f \in I_{\X_0}^k $ and define 
\begin{align*}
\tilde{f} 	:= f t^{-k} \in I_{\X_0}^k t^{-k} 		&\subset \cR	, \\
\init(f) 	:= [f] \in I_{\X_0}^k / I_{\X_0}^{k+1} 	&\subset  \bigoplus_{k \geq 0} ( I_{\X_0}^k / I_{\X_0}^{k+1} ) \cong \O_{\X_0} [s] 
\end{align*}
as local sections of $\cR$ and $ \O_{\X_0} [s] $ respectively. Here, $[f]$ denotes the image of $f\in I_{\X_0}^k$ in $I_{\X_0}^k / I_{\X_0}^{k+1}$.
Moreover, we define the sheaf $ \tilde{I} $ on $\cW$ to be the sheaf of ideals locally generated by $ \{ \tilde{f}  \mid  f \in I \} $ in $ \cR $ and the sheaf $ \init (I) $ on $ \X_0 \times \A^1_s $ to be the sheaf of ideals locally generated by $ \{  \init (f)  \mid  f \in I \} $ in $ \O_{\X_0} [s] $.

\begin{lem}\label{degen-ideal}
In the above setting, the following hold:
\begin{enumerate}
\item We have the equalities
\begin{align*}
\tilde{I} 	&= I[t,t^{-1}] \cap \cR		\\
		&= I[t] + \frac{ I \cap I_{\X_0} }{t} + \frac{ I \cap { I_{\X_0} }^2 }{t^2} + \cdots \subset \cR,	\\
\init(I) 	&= \tilde{I} \cdot \O_{\X_0}[s] \subset \O_{\X_0}[s].
\end{align*}
\item If $ I \subset \O_\X$ defines a closed subscheme $ Z \subset \X $ set-theoretically supported in $\X_0$, then $ \tilde{I} $ defines $ \cZ \subset \cV^{\circ}$ (the strict transform of $Z \times \A^1 \subset \X \times \A^1 $ on $\cV$).
Also, $ \init (I) $ defines $ Z_0 \subset \X_0 \times \A^1_s $.
\item $\cR/\tilde{I}$ is flat as a sheef of $\C[t]$-modules, and so $\cZ$ is flat over $\A^1$.
\end{enumerate}
\end{lem}

\begin{proof}
(1) $\tilde{I} \subset I[t,t^{-1}] \cap \cR$ is clear by the definition.
In order to see the opposite inclusion, it suffices to show that $ft^{-k} \in \tilde{I}$ for any $f \in I \cap I_{\X_0}^k$.
If $\textrm{ord}_{\X_0}(f) = k$, this follows from the definition of $\tilde{I}$. If $\textrm{ord}_{\X_0}(f) > k$, take any $g \in I$ such that $\textrm{ord}_{\X_0}(g) = k$, then we get $ft^{-k} = \widetilde{(f+g)} - \tilde{g} \in \tilde{I}$. 
Thus we obtain the first equality.
The last equality follows since the image of $\tilde{f}$ in $ (\cR / t\cR) \cong \O_{\X_0}[s] $ is $[f] \in I_{\X_0}^k / I_{\X_0}^{k+1}$.

(2) By the first equality in (1), $\tilde{I}$ is the largest ideal in $\cR$ among ideals which coincide with $I[t,t^{-1}]$ when they are extended to $\O_\X[t,t^{-1}]$. So $\tilde{I}$ defines the scheme theoretic closure of $Z \times \C^*$ in $\cV$, which is nothing but $\cZ$.
It also follows that $\init (I) $ defines $ Z_0 $ from the last equality in (1).

(3) is in \cite[Lemma 4.1]{LZ18} and can be proved exactly in the same way as \cite[Lemma 4.1]{LX16}, but here we provide a direct proof.
The flatness is clear outside $0 \in \A_s^1$.
To show the flatness over $0 \in \A_s^1$, we only need to check that $t$ is a non-zero divisor in $\cR/\tilde{I}$, since $(t)$ is the only non-trivial ideal in the base $\C[t]_{(t)}$.
Take any $s \in \cR$ such that $st \in \tilde{I}$. 
Writing down as $s = \sum_i f_i t^{-i}$, we have $f_i \in I_{\X_0}^i$ for $i \ge 0$.
On the other hand $st =  \sum_i f_i t^{-i+1} \in \tilde{I}$ implies $f_i \in I \cap I_{\X_0}^{i-1}$ for $i \ge 1$ and $f_i \in I$ for $i \le 0$.
Combining the above, we get $f_i \in I \cap I_{\X_0}^{i}$ for $i \ge 1$ and $f_i \in I$ for $i \le 0$, which shows $s\in \tilde{I}$.
\end{proof}

\subsection{Comparision of the CM degree}

In this subsection, we show equality of the CM degree of the polarized families under a certain assumption and then discuss when the assumption is satisfied.
We keep the notation in Subsection \ref{construct}.

\begin{prop}\label{CMeq}
Assume that the central fiber $\cW_0$ of $\cW$ in (\ref{totalfamily}) consists of only 2 irreducible components, that is, 
\[
\cW_0 = \hat{\X} \bigcup_{\X_0} \cT.
\]
Then, the equality
\[
\CM( \cT,\cN) = \CM (\cB,\cM) - \CM (\X,\L)
\]
holds.
\end{prop}


\begin{proof}
By flatness and the assumption, we have the equality
\[
\chi(\cT,\cN ) + \chi (\X,\L) - \chi (\X_0, \L_0) = \chi (\cB,\cM).
\]
Comparing the coefficient of $k^{n+1}$ and $k^n$, we get
\begin{align*}
b_0^\cT + b_0^\X = b_0^\cB 	,	\\
b_1^\cT + b_1^\X - a_0 = b_1^\cB,
\end{align*}
where $ b_i^\cT,b_i^\X,b_i^\cB $ are the coefficients of the expansion (\ref{b}) in Definition \ref{CM} for each family.
Notice that the coefficient $a_i$ of the expansion (\ref{a}) in Definition \ref{CM} is the same for each family.
Thus,
\begin{align*}
\CM( \cT, \cN) 	=& a_1 b_0^\cT - a_0 b_1^\cT + a_0^2						\\
			=& a_1 (b_0^\cB -b_0^\X) - a_0 (b_1^\cB - b_1^\X + a_0) + a_0^2		\\
			=& (a_1 b_0^\cB - a_0 b_1^\cB + (1-g(C))a_0^2) 					\\
				&- (a_1 b_0^\X - a_0 b_1^\X + (1-g(C))a_0^2)	\\
			=& \CM (\cB,\cM) - \CM (\X,\L).
\end{align*}
\end{proof}

We give a sufficient condition for the assumption in Proposition \ref{CMeq}.

\begin{lem}\label{lemma3}
Let $I$ be the ideal defining the closed subscheme $Z \subset \X$ supported in $\X_0$.
If $ \widetilde{I^m} = ( \tilde{I} )^m $ holds for any positive integer $m$,
then the central fiber $ \cW_0 $ of $ \cW $ in (\ref{totalfamily}) consists of only 2 components.
\end{lem}
Geometrically, the assumption says that any thickening of $ \cZ $ is still flat over $\A^1$.

\begin{proof}
Define
\[
\cW^{\circ} := \cW \setminus \hat{\X} = \Bl_\cZ \cV^{\circ}	\to 	\A^1_t
\]
so that we need to prove that the central fiber $ \cW^{\circ}_0 $ coincides with the restriction $ \cT|_{\P^1 \setminus \{ \infty \}} = \Bl_{Z_0 \times \0} (\X_0 \times \A^1_s)$.
It is enough to show an isomorphism of $\cR$-algebra
\[
\Bigr( \bigoplus_{m\geq 0}  (\tilde{I})^m \Bigl) \otimes_{\C[t]} \C[t] / (t) \cong  \bigoplus_{m\geq 0} \init(I)^m
.
\]

From the assumption and the flatness, we have
\begin{align*}
(\tilde{I})^m \otimes_{\C[t]} \C[t]/(t)	&=	\widetilde{I^m} \otimes_{\C[t]} \C[t]/(t)	\\
						&\cong	\widetilde{I^m} \cdot \O_{\X_0}[s]			\\
						&=	(\tilde{I})^m \cdot \O_{\X_0}[s]			\\
						&=	\init(I)^m.
\end{align*}
Indeed, the first and the third equalities follow from the assumption and the second follows from the flatness of $\R/\widetilde{I^m}$.
Thus we get the assertion by 
taking the direct sum.

\end{proof}

\section{Proof of the main theorem}

The following lemma is needed to ensure that the assumption in Lemma \ref{lemma3} is satisfied in the setting of Theorem \ref{thm1}.

\begin{lem}\label{lemma2}
Let $A$ be a regular ring essentially of finite type over a field $k$. Assume $(h)\subset A$ is a prime ideal such that $A/(h)$ is also a regular ring and an ideal $I \subset A$ contains $h$. Then, for positive integers $j < m$, $I^m \cap (h^j) = h^j I^{m-j}$ holds.  
\end{lem}

\begin{proof}
The inclusion $ h^j I^{m-j} \subset I^m \cap (h^j) $ is clear, so we prove the opposite inclusion.
First we may assume $A$ is complete by taking completion with respect to its maximal ideal.
Let $ \{ x_2, \cdots, x_n \} $ denote the lift of regular sequence of parameter of $A/(h)$ to $A$, then $\{x_1, x_2, \cdots, x_n \}$ is a regular sequence of parameter of $A$ where we define $x_1 = h$, which induces the isomorphism
$A \cong k[[x_1, x_2, \cdots, x_n]]$ (see \cite[\S 28 the proof of Lemma 1]{Mat}).
So we replace $A$ by a formal power series ring $k[[x_1, x_2, \cdots, x_n]]$ and $h$ by $x_1$.
Then we can write $ I = (x_1, f_1, \cdots,  f_s) $, where each $f_i$ is a formal power series of $x_2, \cdots, x_n$. 
Let $ B = k[[x_2, \cdots, x_n]] $ be a subalgebra of $A$, and 
$J$ be an ideal in $A$ generated by $ f_1, \cdots, f_s $.
Let $ f \in A $ be an element of $ I^m \cap (x_1^j) $.
Since $ f \in I^m$, we can write
\[
f = x_1^m g_0 + x_1^{m-1} g_1 + \cdots + x_1 g_{m-1} + g_m, \ g_i \in J^i.
\]
We may take each $g_i$ from $B$ for $ i > 0$.
Indeed, we can write
\[
g_i = \sum_{1 \leq k_1 \leq \cdots \leq k_i \leq s} f_{k_1} \cdots f_{k_i} F_{\underline{k}}, \ 
F_{\underline{k}} \in A,
\]
where $\underline{k}$ denotes a pair $(k_1,\cdots, k_i)$.
By decomposing as
\[
F_{\underline{k}} = G_{\underline{k}} + x_1H_{\underline{k}}, \ G_{\underline{k}} \in B, H_{\underline{k}} \in A,
\]
we get
\[
g_i = \sum_{1 \leq k_1 \leq \cdots \leq k_i \leq s} f_{k_1} \cdots f_{k_i} G_{\underline{k}}
+x_1 \sum_{1 \leq k_1 \leq \cdots \leq k_i \leq s} f_{k_1} \cdots f_{k_i} H_{\underline{k}}.
\]
The first term is an element of $ B \cap J^i $ since $ f_1, \cdots, f_s $ are elements in $B$.
Replace $ g_i $ by $\sum f_{k_1} \cdots f_{k_i} G_{\underline{k}} \in B $ and $g_{i-1} $ by $g_{i-1}+ \sum f_{k_1} \cdots f_{k_i} H_{\underline{k}} \in J^i$, and repeat this for $ i = m, m-1, \cdots, 1$.
Thus we can assume $ g_i \in B $ for $ i>0 $.
Then we have
$g_i = 0 $ for $i > m-j $ since $ f \in (x_1^j) $.
So we get
\[
f = x_1^{j} ( x_1^{m-j} g_0 + \cdots + g_{m-j}) \in x_1^j I^{m-j}
\]
as desired.
\end{proof}

We now prove Theorem \ref{thm1}.

\begin{proof}[Proof of Theorem \ref{thm1}]

Let $ Z \subset \X_0 $ be a proper closed subscheme and $ c \in (0, \epsilon(Z, \L_0) ] $ be a rational number.
First we assume $ c \in (0, \epsilon(Z, \L_0) ) $.
Take the blow-up
\begin{align*}
\pi_0 : 	&\Bl_{ Z \times \0 } ( \X_0 \times \A^1 ) \to \X_0 \times \A^1	,\\
\pi:		&\Bl_Z  \X \to \X ,
\end{align*}
and denote the Cartier exceptional divisor as $E_0$ and $E$ respectively.
Then, $ (\cT_Z, \L_{Z,c}) := \bigl( \Bl_{ Z \times \0 } ( \X_0 \times \A^1 ) , \pi_0^* p_1^* \L_0 (-cE_0) \bigr)	 \to \A^1 $ is a test configuration by Lemma \ref{lem1}.
Let $ (\cT, \cN) = (\overline{\cT_Z}, \overline{\L_{Z,c}}) \to \P^1 $ be a natural compactification.
Also, $ (\cB,\cM) := \bigl( \Bl_Z \X, \pi^* \L (-cE) \bigr) \to C$ is a polarized family.
Indeed, following the notation in Subsection \ref{construct}, since $\cW $ is a flat family over a surface $ \Bl_{(0,0)} ( C \times \A^1 ) $ (by arguments below) and ampleness is an open condition,  the ampleness of $\cM$ over $C$ follows from that of $\cN$ on $\cT$ over $\P^1$.

Let $ I $ and $ I_{\X_0} $ be the ideal defining $ Z $ and $\X_0$ in $ \X $, respectively.
Then, the equality
\begin{equation}\label{eq1}
 I^m \cap I_{\X_0}^j = I_{\X_0}^j I^{m-j}
\end{equation}
holds for any positive integers $j \leq m$.
Indeed, we may check this locally at a point $x$ in $\X_0$, so let $ A := \O_{x,\X} $. Then $ A $ is regular since $\X_0$ and $C$ are both smooth and $\X \to C$ is flat. Moreover the restriction of $I_{\X_0}$ to $ \Spec A \subset \X$ is a principal prime ideal $ (h) $ of $A$ and the restriction of $I$ is an ideal containing $h$, since $ Z $ is \emph{scheme theoretically} supported in $\X_0$. Since $\X_0$ is smooth, $A/(h)$ is regular.
Thus, we can apply Lemma \ref{lemma2} and get the equality.

Now we construct the deformation of $(\cB,\cM)$ to $ (\cT,\cN) \cup_{\X_0} (\X,\L) $ whose total family is $(\cW, \cF)$ in Subsection \ref{construct}.
Using the algebraic description of the deformation in Subsection \ref{construct}, $ \tilde{I} $ can be written as
\[
\tilde{I} = I[t] + \frac{I_{\X_0}}{t} + \frac{{I_{\X_0}}^2}{t^2} + \cdots \subset \cR,
\]
since $I_{\X_0} \subset I$.  So $ ( \tilde{I} )^k $ can be computed as
\[
( \tilde{I} )^k = I^k[t] + I^{k-1} \frac{I_{\X_0}}{t} + I^{k-2} \frac{ { I_{\X_0} }^2}{t^2} + \cdots \subset \cR.
\]
On the other hand, similarly we can write as follows:
\[
\widetilde{I^k} = I^k[t] + \frac{ I^k \cap I_{\X_0} }{t} + \frac{ I^k \cap { I_{\X_0} }^2 }{t^2} + \cdots \subset \cR.
\]
Thus, we can see that
\[
\widetilde{I^k} = ( \tilde{I} )^k
\]
holds for any positive integer $k$ by the equality (\ref{eq1}).
By Lemma \ref{lemma3}, the central fiber $ \cW_0 $ of the total family $ \cW \to \A^1 $ consists of only 2 irreducible components, and so we can apply Proposition \ref{CMeq} to get the equality
\[
\CM (\cT,\cN) = \CM (\cB,\cM) - \CM (\X,\L) .
\]
But by the assumption on the minimization of CM degree, we have
\[
\CM (\cB,\cM) - \CM (\X,\L) 	\geq 0,
\]
which implies the inequality
\[
\DF (\cT_Z, \L_{Z,c} ) = \CM (\cT, \cN) \geq 0,
\]
where the first equality follows from Remark \ref{DFisCM}.
Thus we can get the desired slope inequality by Theorem \ref{slopeisDF}.

For $ c = \epsilon(Z, \L_0) $, the slope inequality follows from the above argument and the continuity of slope of $Z$ with respect to $c$.
\end{proof}


\end{document}